\documentclass[10pt]{amsart}
\usepackage{amsmath}
\usepackage{amsfonts}
\usepackage{amssymb}
\usepackage{amsthm}
\usepackage{graphicx}

\usepackage[usenames, dvipsnames]{color} 

\usepackage{hyperref}

\def \C {{\mathcal{C}}}

\def \de {\partial}

\def \H {\mathcal{H}}
\def \HH {\textsc{{H}} }
\def \N {\mathbb{N}}
\def \O {\Omega}
\def \phi {\varphi}
\def \RNu {\mathbb{R}^{N+1}}
\def \RN {\mathbb{R}^N}
\def \R {\mathbb{R}}
\def \l {\lambda}

\def \Cyl {\mathcal{C}}

\def \Ockq {\O_{kq}^c(0)}
\def \Tst {T^*_{kq}}

\def\elle{\mathcal{L}}
\def\erre{\mathbb{R}}
\def \emme {{\mathcal {M}}}

\def\erreu{{\erre^{ {N+1} }}}
\def\inn{\mbox{ in }}
\def\capp{\mathrm{cap\, }}
\newcommand{\tende}{\rightarrow}
\newcommand{\ttende}{\longrightarrow}
\newcommand{\enne} {\mathbb{N}}

\newtheorem{theorem}{Theorem}[section]
\newtheorem{lemma}[theorem]{Lemma}

\newtheorem{corollary}[theorem]{Corollary}
\newtheorem{remark}[theorem]{Remark}
\newtheorem{definition}[theorem]{Definition}

\numberwithin{equation}{section}

\begin{document}

\title[Wiener-Landis criterion for Kolmogorov-type operators]{Wiener-Landis criterion \\ for Kolmogorov-type operators}

\author[A.E. Kogoj]{Alessia Kogoj}
\address{Dipartimento di Scienze Pure e Applicate (DiSPeA), 
				 Universit\`{a} degli Studi di Urbino ``Carlo Bo'', 
				 Piazza della Repubblica, 13 - 61029 Urbino (PU), Italy.
         }
 \email{alessia.kogoj@uniurb.it}
\author[E. Lanconelli]{Ermanno Lanconelli}
\address{Dipartimento di Matematica,
         Universit\`{a} degli Studi di Bologna,
         Piazza di Porta S. Donato, 5 - 40126 Bologna, Italy.
         }
 \email{ermanno.lanconelli@unibo.it}
\author[G. Tralli]{Giulio Tralli}
\address{Dipartimento di Matematica,
         Universit\`{a} degli Studi di Bologna,
         Piazza di Porta S. Donato, 5 - 40126 Bologna, Italy.
         }
 \email{giulio.tralli2@unibo.it}


\date{}

\begin{abstract} We establish a necessary and sufficient condition for a boundary point to be regular for the Dirichlet problem related to a class of Kolmogorov-type equations. Our criterion is inspired by two classical criteria for the heat equation: the Evans--Gariepy's Wiener test, and a criterion by Landis expressed in terms of a series of caloric potentials.
\end{abstract}
\maketitle

\section{Introduction}\label{intro}
Aim of this paper is to establish a necessary and sufficient condition for the regularity of a boundary point for the Dirichlet problem related to a class of hypoelliptic evolution equations of Kolmogorov-type. Our criterion is inspired both to the Evans--Gariepy's Wiener test for the heat equation, and to a criterion by Landis, for the heat equation too, expressed in terms of a series of caloric potentials. 

The partial differential operators we are dealing with are of the following type
\begin{equation}\label{KFP}
\elle = \mathrm{div}\left(  A\nabla\right)  +\left\langle B x,\nabla\right\rangle - \partial_t,
\end{equation}
where $A = (a_{i,j})_{i,j = 1, \dots , N}$ and $B= (b_{i,j})_{i,j = 1, \dots , N}$ are $N\times N$ real and constant matrices,  $z=(x,t)=(x_1, \ldots, x_N, t)$ is the point of $\erre^{N+1},$ $\nabla=(\partial_{x_1}, \ldots, \partial_{x_N}),$  $\mathrm{div}$ and $\langle\ , \ \rangle$ stand for the gradient, the divergence and the inner product in 
$\erre^{N}$, respectively.

The matrix $A$ is supposed to be symmetric and positive semidefinite. Moreover, letting
$$E(s): =\exp\left(-sB\right),\quad s\in\erre,$$
we assume that the following Kalman condition is satisfied: the matrix 
$$C(t)=\int_{0}^{t} E(s)AE^T(s)\,ds,$$
 is strictly positive definite for every $t>0$. As it is quite well known, the condition $C(t)>0$ for $t>0$ is equivalent to the hypoellipticity of $\elle$ in \eqref{KFP}, i.e to the smoothness of $u$ whenever $\elle u$ is smooth (see, e.g., \cite{LP}).  We also assume the operator $\elle$ to be homogeneous of degree two with respect to a group of dilations in $\erreu$. As we will recall in Section \ref{operatori}, this is equivalent to assume $A$ and $B$ taking the blocks form \eqref{A} and \eqref{B}.
 
 Under the above assumptions, one can apply results and techniques from potential theory in abstract Harmonic Spaces, as presented, e.g, in \cite{Cin}.  As a consequence, for every bounded open set $\Omega\subseteq \erreu$ and for every function $f\in C(\partial\Omega,\erre)$, the Dirichlet problem 
 \begin{equation}\label{DP} 
 \elle   u= 0 \mbox{ in } \Omega, \qquad 
  u|_{\partial \Omega} = f, 
\end{equation}
has a {\it generalized solution} $H_f^\Omega$ in the sense of Perron--Wiener--Brelot--Bauer. The function $H_f^\Omega$ is smooth and solves the equation in \eqref{DP} in the classical sense. However, it may occur that $H_f^\Omega$ does not assume the boundary datum. A point $z_0\in\partial\Omega$ is called {\it $\elle$-regular} for $\Omega$ if 

$$\lim_{z\tende z_0} H_f^\Omega(z)=f (z_0)\quad\forall\ f \in C(\partial\Omega, \erre).$$ 

Aim of this paper is to obtain a characterization of the $\elle$-regular boundary points in terms of a serie involving $\elle$-potentials of regions in $\Omega^c$, the complement of $\Omega$, within different level sets of $\Gamma$, the fundamental solution of $\elle$. More precisely, if $z_0\in\partial\Omega$ and $\lambda\in ]0,1[$ are fixed, we define for $k\in\N$
$$\O_{k}^c(z_0)=\left\{z \in \Omega^c \,:\,\left(\frac{1}{\l}\right)^{k\log{k}}\leq\Gamma(z_0,z) \leq\left(\frac{1}{\l}\right)^{(k+1)\log{(k+1)}}\right\} \cup \{z_0\}.$$
Then, our main result is the following
\begin{theorem}\label{iff} Let $\Omega$ be a bounded open  subset of $\erreu$ and let $z_0\in\partial\Omega.$ Then $z_0$ is $\elle$-regular for $\de \Omega$  if and only if
\begin{equation} \label{iffeqn} \sum_{k=1}^{\infty}{V_{\Omega^c_k(z_0)}(z_0)}=+\infty.\end{equation} 
\end{theorem}
Here and in what follows,  if $F$ is a compact subset of $\erreu$, $V_F$ will denote the {\it $\elle$-equilibrium potential} of $F$, and $\mathrm{cap\,}(F)$ will denote its {\it $\elle$-capacity}. We refer to Section \ref{recalls} for the precise definitions.

From Theorem \ref{iff}, one easily obtains a corollary resembling the Wiener test for the classical Laplace and Heat operators. 
\begin{corollary}\label{corol} Let $\Omega$ be a bounded open subset of $\erreu$ and $z_0 \in \de\Omega$. The following statements hold:
\begin{itemize} \item[$(i)$] if 
$$\sum_{k=1}^{\infty}{\frac{\mathrm{cap\, } (\Omega^c_k(z_0))}{\l^{k\log{k}}}}=+\infty$$ 
then $z_0$ is $\elle$-regular;
\item[$(ii)$]  if $z_0$ is $\elle$-regular then 
$$\sum_{k=1}^{\infty}{\frac{\mathrm{cap\, } (\Omega^c_k(z_0))}{\l^{(k+1)\log{(k+1)}}}}=+\infty.$$
\end{itemize}
\end{corollary}

We can make the sufficient condition for the $\elle$-regularity more concrete and more geometrical with the following corollary.
\begin{corollary}\label{corol2} Let $\Omega$ be a bounded open subset of $\erreu$ and $z_0 \in \de\Omega$. If
$$\sum_{k=1}^{\infty}{\frac{|\Omega^c_k(z_0)|}{\,\,\l^{\frac{Q+2}{Q}k\log{k}}\,\,}}=+\infty$$
then $z_0$ is $\elle$-regular. In particular, the $\elle$-regularity of $z_0$ is ensured if $\Omega$ has the exterior $\elle$-cone property at $z_0$.
\end{corollary}
If $E$ is a subset of either $\RN$ or $\RNu$, $|E|$ stands for the relative Lebesgue measure. Moreover, $Q$ is the homogeneous dimension recalled in Section \ref{operatori}, and the $\elle$-cone property will be defined precisely in Section \ref{proofw}. We just mention here that it is a natural adaptation of the parabolic cone condition to the homogeneities of the operator $\elle$.

\vskip 0.4cm

\noindent Before proceeding, we would like to comment on Theorem \ref{iff} and Corollary \ref{corol}.\\
A boundary point regularity test for the heat equation involving infinite sum of (caloric) potentials was showed by Landis in \cite{La}. A similar test for a Kolmogorov equation in $\mathbb R^3$ was obtained by Scornazzani in \cite{Sc}. Our Theorem \ref{iff} contains, extends, and improves the criterion in \cite{Sc}. The Wiener test for the heat equation was proved by Evans and Gariepy in \cite{EG}. The extension of such a criterion to the Kolmogorov operators \eqref{KFP} is an open, and seemingly difficult, problem. Our Corollary \ref{corol}, which is a straightforward consequence of Theorem \ref{iff}, is a Wiener-type test giving necessary and sufficient conditions which look ``almost the same''. As a matter of fact, in Theorem \ref{iff} we have considered the $\elle$-potentials of the compact sets $\Omega^c_k(z_0)$ which are built by the difference of two consecutive super-level sets of $\Gamma(z_0,\cdot)$. These level sets correspond with the sequence of values $\lambda^{-k\log{k}}$. The exact analogue of the Evans-Gariepy criterion would have required the sequence with integer exponents $\lambda^{-k}$. The presence of the logarithmic term, which makes the growth of the exponents slightly superlinear, is crucial for our proof of Theorem \ref{iff}. Moreover, such presence is also the responsible for the non-equivalence of the necessary and the sufficient condition in Corollary \ref{corol}. To complete our historical comments, we mention that a potential analysis for Kolmogorov operators of the kind \eqref{KFP} first appeared in \cite{Sc}, in \cite{GL90}, and in \cite{LP}. We also  mention that the cone criterion contained in Corollary  \ref{corol2} has been recently proved
in \cite{K16}, where such a boundary regularity test has been showed for classes of operators more general than  \eqref{KFP}.  For further bibliographical notes concerning Wiener-type tests for both classical and \emph{degenerate} operators, we refer the reader to \cite{LTU}.

The paper is organized as follows. In Section \ref{operatori} we show some structural properties of $\elle$ and fix some notations. Section \ref{recalls} is devoted to the potential theory for $\elle$, while in Section \ref{crucialestimate} a crucial estimate of the ratio between the fundamental solution $\Gamma$ at two different poles is proved.  In Section \ref{necessity} the {\it only if} part of Theorem \ref{iff} is proved. The {\it if} part, the core of our paper, is proved in Section \ref{cap}, where the estimates of Section \ref{crucialestimate} play a crucial r\^ole. Section \ref{proofw} is devoted to the proof of Corollary \ref{corol} and Corollary \ref{corol2}.

\section{Structural properties of $\elle$} \label{operatori}

In \cite[Section 1]{LP} it is proved that the operator $\elle$ is left-translation invariant with respect to the Lie group $\mathbb {K}$ whose underlying manifold is $\mathbb{R}^{N+1}$, endowed with the composition law
$$\left(  x,t\right)  \circ\left(  \xi,s\right)  =\left(  \xi+E(s)  x,t+ s \right).$$

Furthermore, a fundamental solution for $\elle$ is given by   

\[
\Gamma\left(  z,\zeta\right)  =\Gamma \left(  \zeta^{-1}\circ z\right)  \text{
for }z,\zeta\in\mathbb{R}^{N+1},
\]
where, 
\begin{equation}
\Gamma \left(z \right)  =\Gamma \left(x,t\right)  =\left\{
\begin{tabular}
[c]{lll}
$0$ & $\text{for }t\leq0,$\\\mbox{}\\
$\frac{\left(  4\pi\right)  ^{-N/2}}{\sqrt{\det C\left(  t\right)  }}
\exp\left(  -\frac{1}{4}\left\langle C^{-1}\left(  t\right)  x,x\right\rangle - t \, \mathrm{tr} B
\right)  $ & $\text{for }t>0.$
\end{tabular}
\ \right. \nonumber
\end{equation}

We assume the operator $\elle$ to be homogeneous of degree two with respect to a group of dilations. This last assumption, together with the hypoellipticity of $\elle$, implies that the matrices $A$ and $B$ take the following form with respect to some basis of $\mathbb{R}^{N}$ 
(see again \cite[Section 1]{LP}):
\begin{equation}\label{A}
A=
\begin{bmatrix}
A_{0} & 0\\
0 & 0
\end{bmatrix}
\end{equation}
for some $p_{0}\times p_{0}$ symmetric and positive definite matrix $A_{0}$ ($p_{0}\leq N$), and
\begin{equation}\label{B}
B=
\begin{bmatrix}
0 & 0 & \ldots & 0 & 0 \\
B_{1}  & 0  & 0  & 0  & 0  \\
0 & B_{2}& \ldots & 0 & 0 \\
\vdots & \vdots & \ddots & \vdots & \vdots\\
0 & 0 & \ldots &  B_{n} & 0 
\end{bmatrix},
\end{equation}
where $B_{j}$ is a $p_{j-1}\times p_{j}$ block with rank $p_{j}$ ($j=1,2,...,n$), $p_{0}\geq p_{1}\geq...\geq p_{n}\geq1$ and $p_{0}+p_{1}+...+p_{n}=N$. For such a choice we have $\mathrm{tr} B =0$, and the family of automorphisms of $\mathbb{K}$ making $\elle$ homogeneous of degree two can be taken as
\begin{eqnarray*}\delta_r:\erre^{N+1}\ttende\erre^{N+1},\quad \delta_r(x,t)&\,=&
\delta_r(x^{(p_0)},x^{(p_1)},\ldots,x^{(p_n)},t)\\&:=&\left(r
x^{(p_0)},r^3
x^{(p_1)},\ldots,r^{2n+1}x^{(p_n)},r^2t\right),\\ x^{(p_i)}\in\erre^{p_i},\quad i=0,\ldots,n ,\quad r>0.\end{eqnarray*}
We denote by $Q+2 \ (=p_{0}+ 3 p_{1}+...+ (2n +1 ) p_{n} + 2 )$ the {\it homogeneous dimension} of $\mathbb{K}$ with respect to $(\delta_r)_{r>0}$. We explicitly remark 
that $Q$ is the homogenous dimension of $\erre^N$ with respect to the dilations 
 \begin{eqnarray*}D_r:\erre^{N}\ttende\erre^{N},\quad D_r(x) = \left(r
x^{(p_0)},r^3
x^{(p_1)},\ldots,r^{2n+1}x^{(p_n)}\right).\end{eqnarray*}

Under these notations, the matrix $C(t)$ and the fundamental solution of $\elle$ with pole at the origin can be written as follows (\cite[Proposition 2.3]{LP}, see also \cite{K82}):
$$C(t)= D_{\sqrt{t}} C(1) D_{\sqrt{t}} $$
and
\begin{equation}
\Gamma \left(x,t\right)  =\left\{
\begin{tabular}
[c]{lll}
$0$ & $\text{for }t\leq0,$\\ \mbox{}\\
$\frac{c_N}{t^{\frac{Q}{2}}} \exp \left(  -\frac{1}{4}  \langle C^{-1} (1) D_{\frac{1}{\sqrt{t}}}  x,  D_{\frac{1}{\sqrt{t}}}  x \rangle \right)  $ & $\text{for } t>0.$
\end{tabular}
\ \right. \nonumber
\end{equation}
We observe that $\Gamma$ is $\delta_r$-homogeneous of degree $-Q$.\\
Throughout the paper we denote by $\left|\cdot\right|$ the Euclidean norms in $\R^N$, $\R^{p_k}$ or $\R$. We also denote, for $x\in\RN$,
$$\left|x\right|_C^2:=\frac{1}{4}\left\langle C^{-1}(1)x,x\right\rangle.$$
For all $x\in\R^N$, we have 
\begin{equation}\label{cnonc}
\left|E(1)x\right|_C^2\geq \sigma_C^2 \left|x\right|^2
\end{equation}
where $4\sigma_C^2$ is the smallest eigenvalue of the positive definite matrix $E^T(1)C^{-1}(1)E(1)$. We recall that the homogeneous norm $\left\|\cdot\right\|:\RN\longrightarrow\R^+$ is a $D_\lambda$-homogeneous function of degree $1$ defined as follows
$$\left\|x\right\|=\sum_{i=0}^n\left|x^{(p_i)}\right|^{\frac{1}{2i+1}},\qquad\mbox{for }x=\left(x^{(p_0)},\ldots,x^{(p_n)}\right)\in\R^{p_0}\times\ldots\R^{p_n}=\RN.$$
We call homogeneous cylinder of radius $r>0$ centered at $0$ the set
$$\Cyl_{r}: =\left\{x\in\RN\,:\,\left\|x\right\|\leq r\right\}\times\left\{t\in\R\,:\,\left|t\right|\leq r^2\right\}=\delta_{r}\left(\Cyl_1\right),$$
and define $\Cyl_r(z_0): =z_0\circ\Cyl_r$. 

\begin{remark}
The norms $\left\|\cdot\right\|$ and $\left|\cdot\right|$ can be compared as follows
\begin{equation}\label{confrhomnonhom}
\sigma\min{\left\{\left|x\right|, \left|x\right|^{\frac{1}{2n+1}}\right\}}\leq\left\|x\right\|\leq (n+1)\max{\left\{\left|x\right|, \left|x\right|^{\frac{1}{2n+1}}\right\}}\qquad \forall\,x\in\RN,
\end{equation}
where $\sigma=\min_{\left|x\right|=1}{\left\|x\right\|}$.\\
Indeed, on one side we simply have
$$\left\|x\right\|\leq \sum_{i=0}^n\left|x\right|^{\frac{1}{2i+1}}\leq (n+1)\max{\left\{\left|x\right|, \left|x\right|^{\frac{1}{2n+1}}\right\}}\qquad \forall\,x\in\RN.$$
On the other hand, for any $x\neq 0$, we get
$$\frac{\left\|x\right\|}{\min{\left\{\left|x\right|,\left|x\right|^{\frac{1}{2n+1}}\right\}}}\geq\sum_{i=0}^n{\frac{\left|x^{(p_i)}\right|^{\frac{1}{2i+1}}}{\left|x\right|^{\frac{1}{2i+1}}}}=\sum_{i=0}^n{\left|\left(\frac{x}{\left|x\right|}\right)^{(p_i)}\right|^{\frac{1}{2i+1}}}=\left\|\frac{x}{\left|x\right|}\right\|\geq\sigma.$$
\end{remark}

\section{Some recalls from Potential Theory for $\elle$: \\ $\elle$-potentials and $\elle$-capacity}\label{recalls} 
We briefly collect here some notions and results from Potential Theory applied to the operator $\elle$. 

For every open set $\Omega\subseteq \erreu$ we denote 
\begin{equation*} \elle(\Omega): = \{ u\in C^\infty(\Omega) \ | \ \elle u=0 \}.
\end{equation*}
and we call $\elle$-{\it harmonic}  in $\Omega$ the functions in $\elle(\Omega).$

We say that a bounded open set $V\subseteq \Omega$ is
{\it $\elle$-regular} if  for every continuous function $\varphi : \partial V \longrightarrow \erre$, there exists a unique
function, $h_\varphi^V$ in $\elle(V)$, continuous in $\overline V$, such that  
$$h_\varphi^V|_{\partial V}=\varphi.$$  
Moreover,  if $\varphi\geq 0$ then $h_\varphi^V \geq 0$ by the minimum principle. 

A function $u: \Omega \ttende ]-\infty, \infty]$ is called {\it  $\elle$-superharmonic in $\Omega$} if 
 \begin{itemize}

\item[$(i)$] $u$ is lower semi-continuous and $u<\infty$ in a dense subset of $\Omega$;

\item[$(ii)$] for every regular set $V$, $\overline V\subseteq \Omega$, and  for every $\varphi \in C( \partial V, 
\erre)$, 
$\varphi \le u|_{\partial V} $,  it follows  $ u \geq h_\varphi^V$ in $V.$
\end{itemize}
We will denote by  $\overline\elle(\Omega)$  the family  of the $\elle$-superharmonic functions in $\Omega$. Since the operator $\elle$ endows $\erreu$ with a structure of $\beta$-harmonic space satisfying the Doob convergence property (see \cite{Manfr, Cin, K16}), 
by the Wiener resolutivity theorem,  for every $f\in C(\de\Omega)$, the Dirichlet problem
\begin{equation*} 
\begin{cases}
 \elle   u= 0 \mbox{ in } \Omega    \\
  u|_{\partial \Omega} = f \end{cases}
\end{equation*}
has a {\it generalized solution in the sense of Perron--Wiener--Bauer--Brelot} given by

$$H_f^\Omega:= \inf\{ u\in {\overline{\elle}} (\Omega) \ | \ \liminf_{\Omega\ni z\tende \zeta} u(z) \geq f(\zeta)\quad \forall\ \zeta\in \partial\Omega\}.$$

The function $H_f^\Omega$ is $C^\infty(\Omega)$ and satisfies $\elle u=0$ in $\Omega$ in the classical sense. However, it is not true, in general, that $H_f^\Omega$ continuously takes the boundary  values prescribed by $f$. A point $z_0\in\partial\Omega$ such that 
 $$\lim_{\Omega\ni z\tende z_0} H_f^\Omega(z)=f(z_0)\quad\mbox{for every }  f\in C(\partial\Omega)$$ is called {\it $\elle$-regular} for $\Omega$.

For our regularity criteria we still need a few more definitions. We denote by $\emme(\erreu)$ the collection of all nonnegative Radon measure on $\erreu$ and we call 
$$\Gamma_\mu(z) :=\int_\erreu \Gamma(z,\zeta)\ d\mu(\zeta),\qquad z\in \erreu,$$ the {\it $\elle$-potential of $\mu$}.

If $F$ is a compact set of $\erreu$ and $\emme(F)$ is the collection of all nonnegative Radon measure on $\erreu$ with support in $F$,
the {\it $\elle$-capacity of $F$} is defined as
$$\mathrm{cap\, } (F) := \sup\{ \mu(\erreu)\ | \ \mu\in\emme(F), \ \Gamma_\mu \le 1 \mbox{ on } \erreu\}.$$
We list some properties of the $\elle$-capacities $\mathrm{cap}$. For every $F$, $F_1$ and $F_2$ compact subsets of $\erreu$, we have:
\begin{itemize} 
\item[$(i)$] $\capp (F) < \infty $;
\item[$(ii)$] if $F_1\subseteq F_2$, then $\capp (F_1) \le \capp (F_2)$;
\item[$(iii)$] $\capp (F_1 \cup F_2) \le \capp(F_1) + \capp (F_2)$;
\item[$(iv)$] $\capp (z_0 \circ F) = \capp ( F) $ for every $z_0\in\erreu$; 
\item[$(v)$] $\capp (\delta_r (F)) = r^{Q} \capp ( F)$ for every $r>0$;
\item[$(vi)$] if $F=A\times\{\tau\}$ for some compact set $A\subset\RN$, then $\capp (F)=|A|$;
\item[$(vii)$] if $F\subset \RN\times [a,b]$, then we have \begin{equation}\label{capmes}\capp (F)\geq \frac{|F|}{b-a}.\end{equation}
\end{itemize} 
The properties $(i)-(v)$ are quite standard, and they follow from the features of $\Gamma$. We want to spend few words on the last two properties. Property $(vi)$ was proved in \cite[Proposizione 5.1]{L73} in the case of the heat operator, namely with the capacity build on the Gauss-Weierstrass kernel. It can be proved verbatim proceeding in our situation: the main tools are the facts that $\Gamma$ has integral $1$ over $\RN$, and it reproduces the solutions of the Cauchy problems. Property $(vii)$ appears to be new even in the classical parabolic case (at least to the best of our knowledge), and it can be deduced readily from $(vi)$. As a matter of fact, if a compact set $F$ lies in a strip $\RN\times [a,b]$, we have 
$$(b-a)\,\capp(F)=\int_{a}^{b}{\capp(F)\,{\rm d}\tau}\geq\int_{a}^{b}{\capp(F\cap\{t=\tau\})\,{\rm d}\tau}=\int_{a}^{b}{|F\cap\{t=\tau\}|\,{\rm d}\tau}=|F|.$$

The last notions we need are the ones of {\it reduced function} and of {\it balayage} of $1$ on $F$. They are respectively defined by
$$W_F:= \inf \{ v\ | v \in {\overline\elle} (\erreu) ,\ v\geq 0 \inn \erreu,\  v\geq 1 \inn F\},$$
and 
$$V_F(z)=\liminf_{\zeta\ttende z} W_F(\zeta),\qquad z\in \erreu.$$
From general balayage theory we have that $V_F$ is less or equal than $1$ everywhere, identically $1$ in the interior of $F$, it vanishes at infinity, is a superharmonic function on $\erreu$ and harmonic on $\erreu\backslash \partial F$. Furthermore, the following properties will be useful for us. Let $F, F_1, F_2$ be compact subsets of $\erreu$, and let $(F_n)_{n\in\enne}$ be a sequence of compact subsets of $\erre^{N+1}$, we have:
\begin{itemize} 
\item[$(i)$] if $F_1\subseteq F_2\subseteq \erre^{N+1}$, then $V_{F_1}\le V_{F_2};$
\item[$(ii)$] $V_{F_1 \cup F_2} \le V_{F_1}+ V_{F_2};$
\item[$(iii)$]  if $F\subseteq \bigcup_{n\in\enne} F_n$, then $V_F\le \sum_{n=1}^\infty V_{F_n}.$
\end{itemize} 
The first property is a consequence of the definition of balayage; for the second and the third one we refer respectively to \cite[Proposition 5.3.1]{CC} and \cite[Theorem 4.2.2 and Corollary 4.2.2]{CC}.

Now, following the same lines of the proof of \cite[Teorema 1.1]{L73}, we have the existence of a unique measure $\mu_F\in \emme(F)$ such that 
\begin{equation}\label{reprgamma}
V_F(z)=\Gamma_{\mu_F}(z)=\int_\erreu{\Gamma(z,\zeta)\,\rm{d}\mu_F(\zeta)}\quad \forall \, z\in\R^{N+1},
\end{equation}
and $$\mu_F(\erreu)=\mathrm{cap\, } (F).$$
$V_F$ is also called the $\elle$-{\it equilibrium potential} of $F$ and $\mu_F$ the $\elle$-{\it equilibrium measure} of $F$. The proof of this fact relies on the good behavior of $\Gamma$, a representation formula of Riesz-type 
for $\elle$-superharmonic functions proved in \cite[Theorem 5.1]{Cin}, and a Maximum Principle for $\elle$ (see \cite[Proposition 2.3]{Cin}).

Fix now a bounded open set $\O$ compactly contained in $\RNu$, and $z_0=(x_0,t_0)\in\de\O$. Let us denote by
$$G_r=\left\{(x,t)\in \Cyl_r(z_0)\smallsetminus \O\,:\,t\leq t_0\right\}.$$
From general balayage theory and proceeding, e.g., as in \cite [Theorem 4.6]{LU}, we can characterize the regularity of the boundary point of $\Omega$ by the following condition:
\begin{center}
the point $z_0\in\partial\Omega$ is $\elle$-regular if and only if
\begin{equation}\label{regularpoint}
\lim_{r\tende 0} V_{G_r} (z_0) > 0.
\end{equation}
\end{center}

\section{A crucial estimate}\label{crucialestimate}
We start by recalling the following identity, whose proof can be found in \cite[Remark 2.1]{LP} (see also \cite{K82}),
\begin{equation}\label{commu}
E(\lambda^2 s)D_\lambda=D_\lambda E(s)\qquad \forall \lambda>0,\, \forall s\in\erre.
\end{equation}

In what follows we will need the following lemma.
\begin{lemma}\label{matinv}
For $0>t>\tau$ we have the following matrix inequality
$$E^T(t)C^{-1}(t-\tau)E(t)\geq C^{-1}(-\tau).$$
\end{lemma}
\begin{proof}
Since for symmetric positive definite matrices we have
$$M_1\leq M_2\quad \Rightarrow\quad M_1^{-1}\geq M_2^{-1}$$
(see \cite[Corollary 7.7.4]{HJ}) and recalling that $E^{-1}(t)=E(-t)$, it is enough to prove that
\begin{equation}\label{matrixclaim}
E(-t)C(t-\tau)E^T(-t)\leq C(-\tau).
\end{equation}
From the very definition of the matrix $C$ we get
\begin{eqnarray*}
E(-t)C(t-\tau)E^T(-t)&=&e^{tB}\left(\int_{0}^{t-\tau}{e^{-sB}Ae^{-sB^T}}\,\rm{d}s\right)e^{tB^T}=\int_{0}^{t-\tau}{e^{(t-s)B}Ae^{(t-s)B^T}}\,\rm{d}s\\
&=&\int_{-t}^{-\tau}{e^{-\sigma B}Ae^{-\sigma B^T}}\,\rm{d}\sigma.
\end{eqnarray*}
Since $-\tau>-t>0$ and $A$ is nonnegative definite, we have 
$$\int_{-t}^{-\tau}{e^{-\sigma B}Ae^{-\sigma B^T}}\,\rm{d}\sigma\leq \int_{0}^{-\tau}{e^{-\sigma B}Ae^{-\sigma B^T}}\,\rm{d}\sigma = C(-\tau)$$
which proves \eqref{matrixclaim} and the lemma.
\end{proof}

A crucial role in the  proof  of our main theorem will be played by the ratio $\frac{\Gamma(z,\zeta)}{\Gamma(0,\zeta)}$, for $z=(x,t)$ and $\zeta=(\xi,\tau)$ with $0>t>\tau$. We use the following notations
$$\mu=\frac{-t}{-\tau}\in(0,1),\qquad M(z)=\left|D_{\frac{1}{\sqrt{-t}}}x\right|,\qquad M(\zeta)=\left|D_{\frac{1}{\sqrt{-\tau}}}\xi\right|.$$

\begin{lemma}\label{rapporto}
There exists a positive constant $C$ such that, for any $z=(x,t),\zeta=(\xi,\tau)$ with $0>t>\tau$ and $\mu\leq\min{\{\frac{1}{2},\frac{\sigma^2}{(n+1)^2}\}}$, we have
$$\frac{\Gamma(z,\zeta)}{\Gamma(0,\zeta)}\leq\left(\frac{1}{1-\mu}\right)^{\frac{Q}{2}}e^{C\sqrt{\mu}M(z)M(\zeta)}.$$
\end{lemma}
\begin{proof} In our notations we can write
\begin{eqnarray*}
\frac{\Gamma(z,\zeta)}{\Gamma(0,\zeta)}&=&\frac{(t-\tau)^{-\frac{Q}{2}}e^{-\left|D_{\frac{1}{\sqrt{t-\tau}}}\left(x-E(t-\tau)\xi\right)\right|_C^2}}{(-\tau)^{-\frac{Q}{2}}e^{-\left|D_{\frac{1}{\sqrt{-\tau}}}\left(E(-\tau)\xi\right)\right|_C^2}}\\
&=&\left(\frac{1}{1-\mu}\right)^{\frac{Q}{2}}e^{\left|D_{\frac{1}{\sqrt{-\tau}}}\left(E(-\tau)\xi\right)\right|_C^2-\left|D_{\frac{1}{\sqrt{t-\tau}}}\left(x-E(t-\tau)\xi\right)\right|_C^2}.
\end{eqnarray*}
Let us deal with the exponential term
\begin{eqnarray}\label{exponent}
&&\left|D_{\frac{1}{\sqrt{-\tau}}}\left(E(-\tau)\xi\right)\right|_C^2-\left|D_{\frac{1}{\sqrt{t-\tau}}}\left(x-E(t-\tau)\xi\right) \right|_C^2=\\
&=&\frac{1}{4}\left\langle C^{-1}(-\tau)E(-\tau)\xi,E(-\tau)\xi\right\rangle-\frac{1}{4}\left\langle C^{-1}(t-\tau)\left(x-E(t-\tau)\xi\right),\left(x-E(t-\tau)\xi\right)\right\rangle.\nonumber
\end{eqnarray}
Lemma \ref{matinv} says in particular that we have
$$\left\langle C^{-1}(-\tau)E(-\tau)\xi,E(-\tau)\xi\right\rangle-\left\langle C^{-1}(t-\tau) E(t-\tau)\xi,E(t-\tau)\xi\right\rangle\leq 0.$$
Using this in \eqref{exponent} we get
\begin{eqnarray}\label{exponent2}
&&\left|D_{\frac{1}{\sqrt{-\tau}}}\left(E(-\tau)\xi\right)\right|_C^2-\left|D_{\frac{1}{\sqrt{t-\tau}}}\left(x-E(t-\tau)\xi\right) \right|_C^2\leq\\
&\leq&-\frac{1}{4}\left\langle C^{-1}(t-\tau)x,x\right\rangle+\frac{1}{2}\left\langle C^{-1}(t-\tau)x,E(t-\tau)\xi\right\rangle\leq\frac{1}{2}\left\langle C^{-1}(t-\tau)x,E(t-\tau)\xi\right\rangle\nonumber\\
&\leq&\frac{1}{2}\left(\left\langle C^{-1}(t-\tau)x,x\right\rangle\left\langle C^{-1}(t-\tau)E(t-\tau)\xi,E(t-\tau)\xi\right\rangle\right)^{\frac{1}{2}}.\nonumber
\end{eqnarray}
We are going to bound $\left\langle C^{-1}(t-\tau)x,x\right\rangle$ and $\left\langle C^{-1}(t-\tau)E(t-\tau)\xi,E(t-\tau)\xi\right\rangle$ separately. We have
$$\left\langle C^{-1}(t-\tau)x,x\right\rangle=\left\langle C^{-1}\left(\frac{1}{\mu}-1	\right)D_{\frac{1}{\sqrt{-t}}}x,D_{\frac{1}{\sqrt{-t}}}x\right\rangle\leq \left\|C^{-1}\left(\frac{1}{\mu}-1\right)\right\|M^2(z),$$
where $\left\|A\right\|$ stands for the operator norm of a matrix $A$ (i.e. its biggest eigenvalue for symmetric matrices). By \eqref{confrhomnonhom}, for any vector $v$ with $\left|v\right|=1$ we get
$$\min{\left\{\left|D_{\sqrt{\mu}}v\right|,\left|D_{\sqrt{\mu}}v\right|^{\frac{1}{2n+1}}\right\}}\leq\frac{1}{\sigma}\sqrt{\mu}\left\|v\right\|\leq\frac{n+1}{\sigma}\sqrt{\mu}\max{\left\{\left|v\right|,\left|v\right|^{\frac{1}{2n+1}}\right\}}=\frac{n+1}{\sigma}\sqrt{\mu}.$$
From $\mu\leq \frac{\sigma^2}{(n+1)^2}$ we then deduce $\left|D_{\sqrt{\mu}}v\right|\leq \frac{n+1}{\sigma}\sqrt{\mu}$. Hence, since $\mu$ is also less than $\frac{1}{2}$,
\begin{eqnarray*}
\left\langle C^{-1}\left(\frac{1}{\mu}-1\right)v,v \right\rangle&=&\left\langle C^{-1}(1-\mu)D_{\sqrt{\mu}}v,D_{\sqrt{\mu}}v\right\rangle\leq \left\|C^{-1}(1-\mu)\right\|\left|D_{\sqrt{\mu}}v\right|^2\\
&\leq& \frac{(n+1)^2}{\sigma^2}\left\|C^{-1}(1-\mu)\right\|\mu\leq \frac{(n+1)^2}{\sigma^2}\left\|C^{-1}\left(\frac{1}{2}\right)\right\|\mu \quad \forall \left|v\right|=1.
\end{eqnarray*}
This gives
\begin{equation}\label{bx}
\left\langle C^{-1}(t-\tau)x,x\right\rangle\leq \frac{(n+1)^2}{\sigma^2}\left\|C^{-1}\left(\frac{1}{2}\right)\right\|\mu M^2(z).
\end{equation}
On the other hand, by the commutation property \eqref{commu}, we get
\begin{eqnarray*}
&&\left\langle C^{-1}(t-\tau)E(t-\tau)\xi,E(t-\tau)\xi\right\rangle=\left\langle C^{-1}(1-\mu)D_{\frac{1}{\sqrt{-\tau}}}E(t-\tau)\xi,D_{\frac{1}{\sqrt{-\tau}}}E(t-\tau)\xi\right\rangle\\
&\leq&\left\|C^{-1}(1-\mu)\right\|\left|D_{\frac{1}{\sqrt{-\tau}}}E(t-\tau)\xi\right|^2=\left\|C^{-1}(1-\mu)\right\|\left|E(1-\mu)D_{\frac{1}{\sqrt{-\tau}}}\xi\right|^2\\
&\leq&\left\|C^{-1}(1-\mu)\right\|\left\|E^T(1-\mu)E(1-\mu)\right\|M^2(\zeta)\\
&\leq&\left\|C^{-1}\left(\frac{1}{2}\right)\right\|\left\|E^T(1-\mu)E(1-\mu)\right\|M^2(\zeta).
\end{eqnarray*}
Since $0<\mu\leq\frac{1}{2}$, the term $\left\|E^T(1-\mu)E(1-\mu)\right\|$ is bounded from above by a universal constant $C_0^2$. Thus we have
\begin{equation}\label{bexi}
\left\langle C^{-1}(t-\tau)E(t-\tau)\xi,E(t-\tau)\xi\right\rangle\leq C_0^2 \left\|C^{-1}\left(\frac{1}{2}\right)\right\|M^2(\zeta).
\end{equation}
Plugging \eqref{bx} and \eqref{bexi} in \eqref{exponent2}, we get
$$\left|D_{\frac{1}{\sqrt{-\tau}}}\left(E(-\tau)\xi\right)\right|_C^2-\left|D_{\frac{1}{\sqrt{t-\tau}}}\left(x-E(t-\tau)\xi\right) \right|_C^2 \leq\frac{C_0}{2}\frac{n+1}{\sigma}\left\|C^{-1}\left(\frac{1}{2}\right)\right\|\sqrt{\mu}M(z)M(\zeta).$$
Therefore
$$\frac{\Gamma(z,\zeta)}{\Gamma(0,\zeta)}\leq\left(\frac{1}{1-\mu}\right)^{\frac{Q}{2}}e^{C\sqrt{\mu}M(z)M(\zeta)}$$
with $C=\frac{C_0}{2}\frac{n+1}{\sigma}\left\|C^{-1}\left(\frac{1}{2}\right)\right\|$.
\end{proof}

\section{Necessary condition for regularity}\label{necessity} 

The characterization in \eqref{regularpoint}, together with the following lemma, will give the necessity of \eqref{iffeqn} in Theorem \ref{iff}.
\begin{lemma}\label{lochiamiamolemma?} 
For every fixed $p\in\enne$, let us split the set $G_r$ as follows
$$G_r=G_r^p\cup G_r^{*p},$$
where
$$G_r^p=\left\{ z\in G_r \  | \  \Gamma(z_0,z) \geq \left( \frac{1}{\lambda} \right)^{p \log p } \right\}\cup\{z_0\},$$
$$\mbox{and }\quad G_r^{*p}=\left\{ z\in G_r \  | \  \Gamma(z_0,z) \le \left( \frac{1}{\lambda} \right)^{p \log p } \right\}.$$
Then, 
$$\lim_{r\ttende 0} V_{G_r}(z_0)= \lim_{r\ttende 0} V_{G_r^{p}}(z_0).$$
\end{lemma} 
\begin{proof} From the monotonicity and subadditivity properties of the balayage, we have 
$$V_{G_r^{p}}(z_0) \le V_{G_r}(z_0)\le V_{G_r^{p}}(z_0) + V_{G_r^{*p}}(z_0).$$
Furthermore, by \eqref{reprgamma},
$$V_{G_r^{*p}}(z_0)\le \left(\frac{1}{\lambda}\right)^{p \log p} \!\!\!  \!\!\! \mathrm{cap\, } ({G^{*p}_r}).$$
On the other hand, from the monotonicity and homogeneity properties of the capacities, it follows 
$$\mathrm{cap\, } ({G^{*p}_r})  \le  \mathrm{cap\, } ({\C_r}(z_0)) =  \mathrm{cap\, } (z_0\circ \delta_r (\C_1)) = r^Q  \mathrm{cap\, } ({\C_1}(z_0)). $$
Hence $\mathrm{cap\, } ({G^{*p}_r})$ goes to zero as $r$  goes to zero. This proves the lemma.
\end{proof}

\begin{proof}[Proof of necessary condition in Theorem \ref{iff}]
Assume
$$\sum_{k=1}^{\infty}{V_{\Omega^c_k(z_0)}(z_0)}<+\infty.$$
We are going to prove the non regularity of the  boundary point $z_0$. The assumption implies that for every $\varepsilon>0$, there exists $p\in \enne$ such that 
\begin{equation*}  \sum_{k=p}^{\infty}{V_{\Omega^c_k(z_0)}(z_0)}<\varepsilon.\end{equation*} 
On the other hand, with the notations of the previous lemma, for any positive $r$
$$G_r^p\subseteq \bigcup_{k=p}^\infty \Omega_k^c(z_0),$$
so that,
$$V_{G_r^p} (z_0) \le \sum_{k=p}^{\infty}{V_{\Omega^c_k(z_0)}(z_0)}<\varepsilon.$$
Then, from Lemma \ref{lochiamiamolemma?}, we get $\lim_{r\tende 0} V_{G_r} (z_0) \le \varepsilon$ for every $\varepsilon>0$, which implies 
$$\lim_{r\ttende 0} V_{G_r}(z_0)= 0.$$  Hence, by \eqref{regularpoint}, the boundary point $z_0$ is not $\elle$-regular.
\end{proof}

\section{Sufficient condition for regularity}\label{cap}
In this section we prove the {\it if} part of Theorem \ref{iff}. This is the core of our main result and requires three lemmas.

\begin{lemma}\label{per0}
Suppose we have a sequence of compact sets $\{F_k\}_{k\in\N}$ in $\RNu$ such that
$$\begin{cases}
F_k\cap F_h=\emptyset & \text{ if }k\neq h,\\
\forall r>0\,\,\, \exists \,\bar{k} \,\,\mbox{ such that }\,\, F_k\subseteq G_r &  \text{for }k\geq\bar{k}.
\end{cases}$$
Suppose also that the following two conditions hold true:
\begin{itemize}

\item[($i$)]\label{i}$$\sum_{k=1}^{+\infty}{V_{F_k}(z_0)}=+\infty;$$
\item[($ii$)]\label{ii}$$\sup_{h\neq k}\sup{\left\{\frac{\Gamma(z,\zeta)}{\Gamma(z_0,\zeta)}\,:\,z\in F_h,\,\zeta\in F_k\right\}}\leq M_0.$$
\end{itemize}
Then we have
$$V_{G_r}(z_0)\geq \frac{1}{2M_0}\quad\, \mbox{for every positive }r.$$
\end{lemma}
\begin{proof}
Let $A>\frac{2}{M_0}$, and fix any $r>0$. Let us pick $m,n\in\N$ with $m<n$ such that
$$\bigcup_{k=m}^n{F_k}\subseteq G_r\quad\mbox{ and }\quad \sum_{k=m}^nV_{F_k}(z_0)\geq A.$$
We are going to denote by $G_{m,n}=\bigcup_{k=m}^n{F_k}$ and by $W_{m,n}(z)=\sum_{k=m}^nV_{F_k}(z)$. We want to estimate $W_{m,n}$ on $G_{m,n}$.\\
Take $z\in G_{m,n}$. We have then $z\in F_h$ for some $h\in\{m,\ldots,n\}$. Of course we have $V_{F_h}(z)\leq 1$. On the other hand, if $k\neq h$ we get
$$V_{F_k}(z)=\int_{F_k}{\Gamma(z,\zeta)\,\rm{d}\mu_k(\zeta)}=\int_{F_k}{\frac{\Gamma(z,\zeta)}{\Gamma(z_0,\zeta)}\Gamma(z_0,\zeta)\,\rm{d}\mu_k(\zeta)}\leq M_0 V_{F_k}(z_0).$$
Hence $V_{F_k}\leq M_0 V_{F_k}(z_0)$ in $F_h$. By the continuity of the equilibrium potentials (outside of their relative compact sets) there exists an open neighborhood $O_h$ of $F_h$ such that
$$V_{F_k}\leq M_0 V_{F_k}(z_0)+\frac{1}{2^k}\quad\,\,\forall\,k\in\{m,\ldots,n\},\,k\neq h.$$
We put $O=\bigcup_h{O_h}$. In $O$ we get
$$W_{m,n}\leq 1+M_0\sum_{k=m}^n{V_{F_k}(z_0)}+\sum_{k=m}^n{\frac{1}{2^k}}\leq 2+M_0\sum_{k=m}^n{V_{F_k}(z_0)}.$$
If we consider the function $v_{m,n}=\frac{1}{2+M_0\sum_{k=m}^n{V_{F_k}(z_0)}}W_{m,n}$, we thus get $v_{m,n}\leq 1$ in $O$. Moreover, the function $v_{m,n}$ is a nonnegative $\H$-superharmonic in $\RNu$, it is $\H$-harmonic in $\RNu\smallsetminus G_{m,n}$, and it vanishes at the infinity. If we take any function $u\in \Phi_{G_{m,n}}$ we have 
$$\begin{cases}
u-v_{m,n}\in\overline{\HH}(\RNu\smallsetminus G_{m,n}),  & \,\\
\liminf_{z\rightarrow \infty}{u(z)-v_{m,n}(z)}\geq 0, &  \,\\
\liminf_{z\rightarrow\zeta\in\de G_{m,n}}{u(z)-v_{m,n}(z)}\geq u(\zeta)-1\geq0.&  \,
\end{cases}$$
The maximum principle infers that $u-v_{m,n}$ has to be nonnegative in $\RNu\smallsetminus G_{m,n}$. On the other hand, $u\geq 1\geq v_{m,n}$ in $G_{m,n}$. Therefore $u\geq v_{m,n}$ in $\RNu$, for every $u\in \Phi_{G_{m,n}}$. This implies that
$$V_{G_{m,n}}(z)\geq v_{m,n}(z)=\frac{W_{m,n}(z)}{2+M_0\sum_{k=m}^n{V_{F_k}(z_0)}}\quad\,\mbox{for all }z\in \RNu.$$
In particular this has to be true at $z=z_0$, i.e.
$$V_{G_{m,n}}(z_0)\geq \frac{\sum_{k=m}^n{V_{F_k}(z_0)}}{2+M_0\sum_{k=m}^n{V_{F_k}(z_0)}}.$$
Since the function $s\mapsto \frac{s}{2+M_0 s}$ is increasing, we deduce
$$V_{G_{m,n}}(z_0)\geq\frac{A}{2+M_0 A}>\frac{1}{2M_0}.$$
This concludes the proof since $V_{G_r}\geq V_{G_{m,n}}$.
\end{proof}

In order to simplify the notations, from now on we assume $z_0=0\in\de\Omega$. This is not restrictive because of the left-invariance property. We want to choose suitably the compact sets $F_k$ of the previous lemma. For any fixed $\l\in(0,1)$, we recall that
$$\O_{k}^c(0)=\left\{z \in \Omega^c \,:\,\left(\frac{1}{\l}\right)^{k\log{k}}\leq\Gamma(0,z) \leq\left(\frac{1}{\l}\right)^{(k+1)\log{(k+1)}}\right\}. $$
Now, we set $\alpha(k)=k\log{k}$ and  denote
$$T_k=\max_{(x,t)\in \O_{k}^c(0)}{-t}=\left(c_N\l^{\alpha(k)}\right)^{\frac{2}{Q}}.$$
We fix $q\in\N$ such that 
\begin{equation}\label{defq}
q\geq q_0:=4+\frac{m}{\log{\left(\frac{1}{\l}\right)}},\quad\mbox{where }m=\max{\left\{2,\frac{Q}{\log{6}},\frac{2\sigma_C^2}{\log{6}},\frac{Q\log{2}}{\log{8}},\frac{2Q\log{(\frac{n+1}{\sigma})}}{\log{8}}\right\}},
\end{equation}
and $\sigma_C, \sigma$ are the constants in \eqref{cnonc} and \eqref{confrhomnonhom}. We also denote by
$$p=1+\left[\frac{q}{2}\right]=1+\mbox{the integer part of }\frac{q}{2}.$$
So $\frac{q}{2}\leq p\leq 1 + \frac{q}{2} <q-1$. For any $k\in\N$ we want to consider the sets
$$\Ockq=\left\{z \in \Omega^c \,:\,\left(\frac{1}{\l}\right)^{\alpha(kq)}\leq\Gamma(0,z) \leq\left(\frac{1}{\l}\right)^{\alpha(kq+1)}\right\}.$$
Moreover, we put
\begin{equation}\label{defFk}
\Ockq=\left(\Ockq\cap \{t\geq -\Tst\}\right)\cup\left(\Ockq\cap \{t\leq -\Tst\}\right): =F_k^{(0)}\cup F_k
\end{equation}
where
$$\Tst=T_{kq+p}=\left(c_N\l^{\alpha(kq+p)}\right)^{\frac{2}{Q}}.$$
First we notice that, since $kq+p<q(k+1)$, $F_k$ lies strictly below $F_{k+1}$, namely
\begin{equation}\label{sotto}
\min_{(x,t)\in F_{h}}{t}=-T_{hq}>-\Tst=\max_{(\xi,\tau)\in F_{k}}{\tau}\qquad\forall h,k\in\N,\,\,h>k.
\end{equation}

\begin{lemma}\label{peri}
We have
$$\sum_{k=1}^{+\infty}{V_{F_k^{(0)}}(0)}<+\infty.$$
\end{lemma}
\begin{proof}
We are going to prove that $F_k^{(0)}$ is contained in a homogeneous cylinder $\Cyl_{r_k}$ so that
\begin{equation}\label{claimi}
\sum_{k=1}^{+\infty}{\left(\frac{1}{\l}\right)^{\alpha(kq+1)}r_k^Q}<+\infty.
\end{equation}
This is enough to prove the statement since
$$V_{F_k^{(0)}}(0)=\int_{F_k^{(0)}}{\Gamma(0,\zeta)\,\rm{d}\mu_{F_k^{(0)}}(\zeta)}\leq\left(\frac{1}{\l}\right)^{\alpha(kq+1)}\mathrm{cap\, } (F_k^{(0)}),$$
and by monotonicity and homogeneity we have $\mathrm{cap\, } (F_k^{(0)})\leq \mathrm{cap\, } (\Cyl_{r_k})= \mathrm{cap\, } (\Cyl_{1}) r_k^Q$. In order to prove \eqref{claimi}, we have to find a good bound for $r_k$.\\
Fix $z=(x,t)\in F_k^{(0)}$. Since in particular $z\in\Ockq$, we have
$$\left|D_{\frac{1}{\sqrt{-t}}}(E(-t)x)\right|_C^2\leq\log{\left(\frac{c_N\lambda^{\alpha(kq)}}{(-t)^{\frac{Q}{2}}}\right)}.$$
On the other hand, by \eqref{commu} and \eqref{cnonc}, we get
$$\left|D_{\frac{1}{\sqrt{-t}}}(E(-t)x)\right|_C^2=\left|E(1)D_{\frac{1}{\sqrt{-t}}}x\right|_C^2\geq\sigma^2_C \left|D_{\frac{1}{\sqrt{-t}}}x\right|^2$$
and then
\begin{equation}\label{boundM}
\left|D_{\frac{1}{\sqrt{-t}}}x\right|^2\leq\frac{1}{\sigma^2_C }\log{\left(\frac{c_N\lambda^{\alpha(kq)}}{(-t)^{\frac{Q}{2}}}\right)}.
\end{equation} 
Therefore, from \eqref{confrhomnonhom}, we deduce
\begin{eqnarray*}
\frac{1}{\sqrt{-t}}\left\|x\right\|&=&\left\|D_{\frac{1}{\sqrt{-t}}}x\right\|\leq (n+1)\max{\left\{\left|D_{\frac{1}{\sqrt{-t}}}x\right|, \left|D_{\frac{1}{\sqrt{-t}}}x\right|^{\frac{1}{2n+1}}\right\}}\\
&\leq&(n+1)\max{\left\{\frac{1}{\sigma_C}\log^{\frac{1}{2}}{\left(\frac{c_N\lambda^{\alpha(kq)}}{(-t)^{\frac{Q}{2}}}\right)}, \frac{1}{\sigma_C^{\frac{1}{2n+1}}}\log^{\frac{1}{2(2n+1)}}{\left(\frac{c_N\lambda^{\alpha(kq)}}{(-t)^{\frac{Q}{2}}}\right)}\right\}}.
\end{eqnarray*}
Let us remark that from our choice $\alpha(k)=k\log{k}$ we can check that the sequence
$$\alpha(kq+p)-\alpha(kq)\mbox{ is monotone increasing}.$$
In particular $\alpha(kq+p)-\alpha(kq)\geq\alpha(\frac{3}{2}q)-\alpha(q)\geq \frac{1}{2}q\log{(\frac{3}{2}q)}\geq\frac{1}{2}q\log{6}$. By our choice of $q$ \eqref{defq}, we have then $\alpha(kq+p)-\alpha(kq)\geq\frac{Q}{2\log{(\frac{1}{\l})}}$ and so
$$(\Tst)^{\frac{Q}{2}}=c_N\lambda^{\alpha(kq+p)}\leq c_N\lambda^{\alpha(kq)}e^{-\frac{Q}{2}}\qquad\forall k.$$
This fact and the fact that the functions $s\mapsto s\log^{\beta}\frac{\alpha}{s^Q}$ are increasing in the interval $(0,e^{-\beta}\alpha^{\frac{1}{Q}}]$ allow to bound the term $\left\|x\right\|$ further. Indeed, having $0<-t\leq\Tst$, we get
\begin{eqnarray*}
\left\|x\right\|&\leq&(n+1)\max{\left\{\frac{\sqrt{-t}}{\sigma_C}\log^{\frac{1}{2}}{\left(\frac{c_N\lambda^{\alpha(kq)}}{(-t)^{\frac{Q}{2}}}\right)}, \frac{\sqrt{-t}}{\sigma_C^{\frac{1}{2n+1}}}\log^{\frac{1}{2(2n+1)}}{\left(\frac{c_N\lambda^{\alpha(kq)}}{(-t)^{\frac{Q}{2}}}\right)}\right\}}\\
&\leq&(n+1)\max{\left\{\frac{\sqrt{\Tst}}{\sigma_C}\log^{\frac{1}{2}}{\left(\frac{c_N\lambda^{\alpha(kq)}}{(\Tst)^{\frac{Q}{2}}}\right)}, \frac{\sqrt{\Tst}}{\sigma_C^{\frac{1}{2n+1}}}\log^{\frac{1}{2(2n+1)}}{\left(\frac{c_N\lambda^{\alpha(kq)}}{(\Tst)^{\frac{Q}{2}}}\right)}\right\}}.
\end{eqnarray*}
Since $\frac{1}{2}q\log{6}\geq \frac{\sigma^2_C}{\log{(\frac{1}{\l})}}$ we have also $(\Tst)^{\frac{Q}{2}}\leq c_N\lambda^{\alpha(kq)}e^{-\sigma_C^2}$, which says $\log^{\frac{1}{2}}{\left(\frac{c_N\lambda^{\alpha(kq)}}{(\Tst)^{\frac{Q}{2}}}\right)}\geq \sigma_C$ and implies
$$\left\|x\right\|\leq\frac{(n+1)}{\sigma_C}\sqrt{\Tst}\log^{\frac{1}{2}}{\left(\frac{c_N\lambda^{\alpha(kq)}}{(\Tst)^{\frac{Q}{2}}}\right)}.$$
Summing up, we have just proved that
$$(x,t)\in F_k^{(0)}\qquad\Longrightarrow\quad\begin{cases}
\left\|x\right\|\leq\frac{n+1}{\sigma_C}\sqrt{\Tst}\log^{\frac{1}{2}}{\left(\frac{c_N\lambda^{\alpha(kq)}}{(\Tst)^{\frac{Q}{2}}}\right)}=:r_k & \text{ and }\\
0<-t\leq\Tst\leq (n+1)^2\Tst\leq r_k^2,&  \,
\end{cases}
$$
namely
$$F_k^{(0)}\subseteq \Cyl_{r_k}.$$
We are left with verifying \eqref{claimi} with this definition of $r_k$. We have thus to prove that
$$\sum_{k=1}^{+\infty}{\left(\frac{1}{\l}\right)^{\alpha(kq+1)-\alpha(kq+p)}\left(\alpha(kq+p)-\alpha(kq)\right)^{\frac{Q}{2}}}<+\infty.$$
The sequences $\alpha(kq+1)-\alpha(kq+p)$ and $\alpha(kq+p)-\alpha(kq)$ are asymptotically equivalent respectively to $(1-p)\log(kq+p)$ and $p\log(kq+p)$. Hence, the series is equivalent to
$$\sum_{k=1}^{+\infty}{\frac{1}{\left(kq+p\right)^{(p-1)\log\frac{1}{\l}}}\log^{\frac{Q}{2}}(kq+p)},$$
which is convergent since $p\geq\frac{q}{2}>1+\frac{1}{\log(\frac{1}{\l})}$. This proves \eqref{claimi}, and therefore the lemma.
\end{proof}

\begin{lemma}\label{perii}
There exists a positive constant $M_0$ such that
$$\frac{\Gamma(z,\zeta)}{\Gamma(0,\zeta)}\leq M_0 \quad\forall\,z\in F_h,\,\forall\,\zeta\in F_k,\quad\forall\,h,k\in\N,\,\,h\neq k.$$
\end{lemma}
\begin{proof}
Fix any $h,k\in \N$ with $h\neq k$. If $h\leq k-1$, then $\Gamma(z,\zeta)=0$ and the statement is trivial. Thus, suppose $h\geq k+1$. For every $z=(x,t)\in F_h$ and $\zeta=(\xi,\tau)\in F_k$ we have
$$\mu=\frac{-t}{-\tau}\leq\frac{T_{hq}}{\Tst}=\left(\frac{\l^{\alpha(hq)}}{\l^{\alpha(kq+p)}}\right)^{\frac{2}{Q}}=\left(\frac{1}{\l}\right)^{\frac{2}{Q}(\alpha(kq+p)-\alpha(hq))}.$$
By monotonicity we have $\alpha(hq)-\alpha(kq+p)\geq\alpha(kq+q)-\alpha(kq+p)\geq\alpha(2q)-\alpha(q+p)\geq\alpha(2q)-\alpha(\frac{3}{2}q+1)\geq(\frac{q}{2}-1)\log{(2q)}$. By our choice of $q$ \eqref{defq} we have then
$$\alpha(hq)-\alpha(kq+p)\geq\left(\frac{q}{2}-1\right)\log{(8)}\geq\frac{Q}{2}\frac{\max{\{\log{2},\log{(\frac{n+1}{\sigma})^2}\}}}{\log{(\frac{1}{\l})}}$$
which implies $\mu\leq \min{\{\frac{1}{2},\frac{\sigma^2}{(n+1)^2}\}}$. This fact allows us to exploit Lemma \ref{rapporto} and get
$$\frac{\Gamma(z,\zeta)}{\Gamma(0,\zeta)}\leq\left(\frac{1}{1-\mu}\right)^{\frac{Q}{2}}e^{C\sqrt{\mu}M(z)M(\zeta)}\leq 2^{\frac{Q}{2}}e^{C\sqrt{\mu}M(z)M(\zeta)},$$
for some structural positive constant $C$. To prove the statement we need to show that the term
$$\mu M^2(z)M^2(\zeta)$$
is uniformly bounded for $z\in F_h$ and $\zeta\in F_k$. By estimating as in \eqref{boundM} we have
\begin{eqnarray*}
M^2(z)&=&\left|D_{\frac{1}{\sqrt{-t}}}x\right|^2\leq\frac{1}{\sigma^2_C }\log{\left(\frac{c_N\lambda^{\alpha(hq)}}{(-t)^{\frac{Q}{2}}}\right)}\\
&\leq& \frac{1}{\sigma^2_C }\log{\left(\frac{c_N\lambda^{\alpha(hq)}}{(T^*_{hq})^{\frac{Q}{2}}}\right)}=\frac{1}{\sigma^2_C }\log{\left(\frac{1}{\l}\right)}(\alpha(hq+p)-\alpha(hq)),
\end{eqnarray*}
and analogously
$$M^2(\zeta)\leq\frac{1}{\sigma^2_C }\log{\left(\frac{1}{\l}\right)}(\alpha(kq+p)-\alpha(kq)).$$
In order to bound $\mu M^2(z)M^2(\zeta)$ we are thus going to estimate the term
\begin{eqnarray*}
&&(\alpha(kq+p)-\alpha(kq))(\alpha(hq+p)-\alpha(hq))\left(\frac{1}{\l^{\frac{2}{Q}}}\right)^{(\alpha(kq+p)-\alpha(hq))}\\
&\leq&(\alpha(kq+p)-\alpha(kq))(\alpha(hq+p)-\alpha(hq))\left(\frac{1}{\l^{\frac{2}{Q}}}\right)^{( \alpha(kq+p)-\alpha(kq+q-1)+\alpha(hq-1)-\alpha(hq))}\\
&=&\left((\alpha(kq+p)-\alpha(kq))\left(\frac{1}{\l^{\frac{2}{Q}}}\right)^{( \alpha(kq+p)-\alpha(kq+q-1))}\right)\left((\alpha(hq+p)-\alpha(hq))\left(\frac{1}{\l^{\frac{2}{Q}}}\right)^{(\alpha(hq-1)-\alpha(hq))}\right)\\
&=:&A_k \cdot B_h.
\end{eqnarray*}
Since $p<q-1$ and $\alpha(n + s)-\alpha(n)$ is asymptotically equivalent to $s\log(n+s)$ as $n$ goes to $\infty$, it is easy to check that the sequences $A_k$ and $B_h$ are convergent to $0$. Therefore they are a fortiori bounded. This proves the lemma.
\end{proof}

\begin{proof}[Proof of sufficient condition in Theorem \ref{iff}] As we noticed, it is not restrictive to assume $z_0=0$. Then, our assumption implies 
$$\sum_{k=1}^{\infty}{V_{\Omega^c_k(0)}(0)}=+\infty$$
for some fixed $\l\in (0,1)$. Take $q\in\N$ as in \eqref{defq}. There exists at least one $i\in\{0,\ldots,q-1\}$ such that
$$\sum_{k=1}^{\infty}{V_{\Omega^c_{kq+i}(0)}(0)}=+\infty.$$
We can assume without loss of generality that $i=0$, i.e. 
$$\sum_{k=1}^{\infty}{V_{\Omega^c_{kq}(0)}(0)}=+\infty.$$
Let us split the sets $\Ockq$ as in \eqref{defFk}. In this way we have defined the sequence of compact sets $F_k$. We want to check that such a sequence satisfies the hypotheses of Lemma \ref{per0}.\\
First of all, from \eqref{sotto}, we have that the $F_k$'s are disjoint. Moreover, since $F_k\subset\Ockq$, it is easy to see that the sets converge from below to the point $0$ (e.g., using that $\Gamma(0,\cdot)$ is $\delta_r$-homogeneous of degree $-Q$). Lemma \ref{perii} provide the existence of a positive constant $M_0$ for which condition (\hyperref[ii]{$ii$}) in Lemma \ref{per0} holds true. The last assumption we have to verify is the condition (\hyperref[i]{$i$}). To do this, we recall that the subadditivity of the equilibrium potentials implies that
$$V_{\Ockq}\leq V_{F_k}+V_{F^{(0)}_k}.$$
Lemma \ref{peri} says that $\sum_{k}{V_{F^{(0)}_k}(0)}$ is convergent. We then deduce
$$\sum_{k=1}^{+\infty}{V_{F_k}(0)}=+\infty,$$
which is condition (\hyperref[i]{$i$}).\\
Then, we can apply Lemma \ref{per0} and infer that $V_{G_r}(0)\geq \frac{1}{2M_0}$ for all positive $r$. The regularity of the point $0$ is thus ensured by the characterization in \eqref{regularpoint}.
\end{proof}

\section{The Wiener-type test, and the cone condition}\label{proofw}

In this section we want prove Corollary \ref{corol}, and Corollary \ref{corol2}.\\
First, we want to show how one can deduce the Wiener-type test of Corollary \ref{corol} from Theorem \ref{iff}: it follows easily from the representation of the potentials \eqref{reprgamma}.
\proof[Proof of Corollary \ref{corol}] For every $k\in \N$, we denote by $\mu_k$ the $\elle$-equilibrium measure of 
$\Omega^c_k(z_0).$ Then, keeping in mind the very definition of $\Omega^c_k(z_0)$, we have:
\begin{eqnarray*} V_{\Omega^c_k(z_0)}(z_0) &=& \int_{\Omega^c_k(z_0)} \Gamma(z_0,\zeta)\ d\mu_k(\zeta)\ d\zeta \\
&\le&  {\left(\frac{1}{\lambda}\right)^{(k+1) \log{(k+1)}}} \mu_k (\Omega^c_k(z_0))= {\frac{\mathrm{cap\, }(\Omega^c_k(z_0))}{\l^{(k+1)\log{(k+1)}}}}.
\end{eqnarray*}
Analogously,
\begin{eqnarray*} V_{\Omega^c_k(z_0)}(z_0) \geq {\frac{\mathrm{cap\, } (\Omega^c_k(z_0))}{\l^{k\log{k}}}}.
\end{eqnarray*}
Hence,
\begin{equation*} \sum_{k=1}^{\infty}{\frac{\mathrm{cap\, } (\Omega^c_k(z_0))}{\l^{k\log{k}}}} \le   \sum_{k=1}^{\infty} V_{\Omega^c_k(z_0)}(z_0) \le 
 \sum_{k=1}^{\infty}{\frac{\mathrm{cap\, } (\Omega^c_k(z_0))}{\l^{(k+1) \log{(k+1)}}}}.
 \end{equation*} 
The assertions $(i)$ and $(ii)$ directly follow from these inequalities, and from Theorem \ref{iff}.
\endproof 

The main statement in Corollary \ref{corol2} follows from the sufficient condition $(i)$ we have just proved, and from \eqref{capmes}. In fact, we have
\begin{equation}\label{capmesterm}
\mathrm{cap\, } (\Omega^c_k(z_0))\geq\frac{|\Omega^c_k(z_0)|}{\left(c_N\lambda^{k\log{k}}\right)^{\frac{2}{Q}}}
\end{equation}
since $\Omega^c_k(z_0)\subset \RN\times [t_0-T_k,t_0]$ where we recall that $T_k^{\frac{Q}{2}}=c_N\lambda^{k\log{k}}$.\\ 
Finally, we have to deal with the proof of the cone condition. To this aim, we need some definitions. We call $\elle$-cone of vertex $0\in\RNu$ a set of the form
$$K_R(B):=\{(D_r(\xi),-r^2)\,:\,\xi\in B,\,\,0\leq r\leq R\}$$
for some bounded set $B\subset \RN$ with non-empty interior, and for some positive $R$. We call $\elle$-cone of vertex $z_0$ the left-translated cone
$$z_0\circ K_R(B).$$
\begin{definition}
Let $\Omega$ be a bounded open subset of $\erreu$ and $z_0 \in \de\Omega$. We say that $\Omega$ has the exterior $\elle$-cone property at $z_0$ if there exists an $\elle$-cone of vertex $z_0$ which is completely contained in $\Omega^c$.
\end{definition}
We can now complete the proof.
\begin{proof}[Proof of Corollary \ref{corol2}]
As we said, from \eqref{capmesterm} we get
$$\sum_{k=1}^{\infty}{\frac{\mathrm{cap\, } (\Omega^c_k(z_0))}{\l^{k\log{k}}}}\geq c_N^{-\frac{2}{Q}} \sum_{k=1}^{\infty}{\frac{|\Omega^c_k(z_0)|}{\,\,\l^{\frac{Q+2}{Q}k\log{k}}\,\,}}$$
and the first part of the proof follows. If we suppose that $\Omega$ has the exterior $\elle$-cone property at $z_0$, we want to prove that the series on the r.h.s. is divergent. In particular, we are going to prove that the terms of that series are uniformly bigger than a positive constant, for $k$ big enough.\\
Without loss of generality, we can assume $z_0=0$. Denote
$$F_r^{\theta}:=\left\{z\in\RNu\,:\, \frac{1}{r}\leq \Gamma(0,z)\leq\frac{\theta}{r}\right\},\quad\mbox{ for }r>0, \mbox{ and for }\theta>1,$$
and let $r_k=\lambda^{k\log{k}}$. For any $\theta>1$ there exists $\bar{k}$ such that we have
$$\Omega^c_k(z_0)\supseteq F_{r_k}^{\theta}\cap K_R(B)\qquad\forall k\geq\bar{k}.$$
On the other hand $$F_{r_k}^{\theta}\cap K_R(B)=\delta_{r_k^{\frac{1}{Q}}}\left(F_1^\theta\cap K_{R\,r_k^{-\frac{1}{Q}}}(B)\right).$$
We claim that there exist $\bar{k}_1\geq\bar{k}$ and a non-empty open set $A\subset \RNu$ such that 
\begin{equation}\label{claimcorol}
A\subseteq F_1^\theta\cap K_{R\,r_k^{-\frac{1}{Q}}}(B)\qquad \forall k\geq\bar{k}_1.
\end{equation}
If this is the case, we get
$$|\Omega^c_k(z_0)|\geq\left|\delta_{r_k^{\frac{1}{Q}}}\left(F_1^\theta\cap K_{R\,r_k^{-\frac{1}{Q}}}(B)\right)\right|=r_k^{\frac{Q+2}{Q}}|F_1^\theta\cap K_{R\,r_k^{-\frac{1}{Q}}}(B)|\geq r_k^{\frac{Q+2}{Q}}|A|$$
for all $k\geq\bar{k}_1$, which is exactly the desired relation
$$\frac{|\Omega^c_k(z_0)|}{\,\,\l^{\frac{Q+2}{Q}k\log{k}}\,\,}\geq |A|>0 \qquad \forall k\geq\bar{k}_1.$$
Hence, we are left with the proof of the claim \eqref{claimcorol}. Take $\bar{k}_1\geq\bar{k}$ such that
$$\sup_{\xi\in {\rm int}(B)}{\Gamma(0,(\xi,-1))}<\frac{R^Q}{r_k}\qquad\forall k\geq\bar{k}_1.$$
Consider 
$$A:=\left\{(D_\rho(\xi),-\rho^2)\,:\,\xi\in {\rm int}(B),\,\,\mbox{and }\frac{1}{\theta}\Gamma(0,(\xi,-1))<\rho^Q<\Gamma(0,(\xi,-1))\right\},$$
which is open, and non-empty since ${{\rm int}}(B)\neq\emptyset$ and $\theta>1$. Moreover $A\subset F^\theta_1$ by construction, and $A\subset K_{R\,r_k^{-\frac{1}{Q}}}(B)$ for $k\geq\bar{k}_1$ because of the inequality $\rho^Q<\frac{R^Q}{r_k}$. The proof is thus complete.
\end{proof}

\section*{Acknowledgments}

A.E.K. and G.T.  have been partially supported by the Gruppo Nazionale per l'Analisi Matematica, la Probabilit\`a e le
loro Applicazioni (GNAMPA) of the Istituto Nazionale di Alta Matematica (INdAM).

\bibliographystyle{amsplain}

\end{document}